\documentclass[12pt]{article}
\usepackage{amsfonts,url,amssymb,amsmath}
\usepackage[margin=2cm]{geometry}
\newcommand{\Aut}{\mathop{\mathrm{Aut}}\nolimits}
\newcommand{\Sym}{\mathop{\mathrm{Sym}}\nolimits}

\newcommand{\ind}{\mathop{\mathrm{ind}}\nolimits}

\newtheorem{theorem}{Theorem}[section]
\newtheorem{prop}[theorem]{Proposition}
\newtheorem{cor}[theorem]{Corollary}

\newtheorem{unnumber}{}

\newtheorem{prepf}[unnumber]{Proof}
\newenvironment{proof}{\prepf\rm}{\endprepf}

\def\wwwatlas{{\sc www-Atlas}\ }
\def\wwwatlasn{{\sc www-Atlas}}
\def\GAP{{\sf GAP}}

\def\xsp#1#2{\mathord{2^{^{_{\,1+#2\!}}}_{^{#1}}}}
\def\mat#1{\mathord{\mathrm{M}_{#1}}}
\def\jan#1{\mathord{\mathrm{J}_{#1}}}
\def\udot{\mathord{{}^{\textstyle{{\cdot}}}{}}}
\def\cn{\mathord{{:}}}
\def\cen{\mathord{\mathrm{C}}}

\def\AAA{\mathord{\mathrm{A}}}
\def\BB{\mathord{\mathrm{B}}}

\def\BBF{\mathord{\mathbb{F}}}

\let\le\leqslant
\let\ge\geqslant

\begin{document}
\title{The Hall--Paige conjecture, and synchronization for affine and diagonal
groups}
\author{John N. Bray\thanks{School of Mathematical Sciences, Queen Mary
University of London, Mile End Road, London E1 4NS, UK} ,
Qi Cai\thanks{No 768, Juxian Street, Chenggong Campus, Yunnan Normal University,
Kunming, Yunnan Province, 651010, China} ,
Peter J. Cameron${}^*$\thanks{School of Mathematics and Statistics, University
of St Andrews, North Haugh, St Andrews, Fife  KY16 9SS, UK} , \\
Pablo Spiga\thanks{Dipartimento di Matematica e Applicazioni, University of
Milano-Bicocca, Milano, 20125 Via Cozzi 55, Italy} , and
Hua Zhang${}^\dag$}
\date{Dedicated to the memory of Charles Sims}
\maketitle

\begin{abstract}
The Hall--Paige conjecture asserts that a finite group has a complete mapping
if and only if its Sylow subgroups are not cyclic. The conjecture is now proved,
and one aim of this paper is to document the final step in the proof (for the
sporadic simple group $\jan{4}$).

We apply this result to prove that primitive permutation groups of simple
diagonal type with three or more simple factors in the socle are
non-synchronizing. We also give the simpler proof that, for groups of affine
type, or simple diagonal type with two socle factors, synchronization and
separation are equivalent.

Synchronization and separation are conditions on permutation groups which are
stronger than primitivity but weaker than $2$-homogeneity, the second of these
being stronger than the first. Empirically it has been found that groups which
are synchronizing but not separating are rather rare. It follows from our
results that such groups must be primitive of almost simple type.

\smallskip
\noindent{\em Keywords: }Automata, complete mappings, graphs, 
Hall--Paige conjecture, orbitals, primitive groups, separating groups,
synchronizing groups, transformation semigroups

\smallskip

\noindent{\em MSC classification: }Primary 20B15; secondary 05E30, 20M35

\end{abstract}

\section{Introduction}

In this section, we recall the definition of synchronization and separation for
permutation groups, and the O'Nan--Scott theorem in the form we require, and
state two theorems which imply that groups which are synchronizing but not
separating must be almost simple. The proof in the case of diagonal groups
requires the truth of the \textit{Hall--Paige conjecture}; the second section
describes this conjecture, and the computations required to prove the final
case needed to resolve it. The final section gives the analysis of diagonal
groups and applies the Hall--Paige conjecture to show that primitive groups
of simple diagonal type with at least three socle factors are non-synchronizing,
and also the (simpler) proof that synchronization and separation are equivalent
for groups of affine type, and of simple diagonal type with two socle factors.

The concept of synchronization arose in automata theory; we state it here for
transformation monoids. A transformation monoid $M$ on a finite set $\Omega$
is \emph{synchronizing} if it contains a transformation of rank~$1$ (one
whose image is a single point).

A permutation group cannot be synchronizing in this sense unless $|\Omega|=1$;
so by abuse of language we redefine the term, and say that the permutation
group $G$ is \emph{synchronizing} if, for every transformation $t$ of $\Omega$
which is not a permutation, the monoid $M=\langle G,t\rangle$ is synchronizing
in the preceding sense.

The definition can be re-phrased in a couple of ways, the first in traditional
permutation group language, the second in terms of graphs. The \emph{clique
number} of a graph is the size of the largest complete subgraph; and the
\emph{chromatic number} is the smallest number of colours required to colour
the vertices so that adjacent vertices are given different colours (this is
called a \emph{proper colouring} of the graph). Since the vertices of a
complete subgraph must all have different colours in a proper colouring, we
see that the clique number is not greater than the chromatic number.

\begin{theorem}
Let $G$ be a permutation group on $\Omega$.
\begin{itemize}
\item[(a)] $G$ is non-synchronizing if and only if there is a non-trivial
partition $P$ of $\Omega$ and a subset $A$ of $\Omega$ such that, for all
$g\in G$, $Ag$ is a transversal for $P$. (We will say that the pair $(A,P)$
\emph{witnesses} non-synchronization.)
\item[(b)] $G$ is non-synchronizing if and only if there is a graph $\Gamma$
on the vertex set $\Omega$, not complete or null, with clique number equal
to chromatic number, such that $G\le\Aut(\Gamma)$.
\end{itemize}
\end{theorem}

We note that a synchronizing group must be primitive, since if there is a
fixed non-trivial partition $P$, then $P$ and any transversal $A$ witness 
non-synchronization. Similarly, a $2$-homogeneous group is synchronizing,
since it is not contained in the automorphism group of a non-trivial graph.

The related concept of separation has no connection with automata, but has
proved very useful in studying synchronization. We note first that a simple
counting argument shows that, if $A$ and $B$ are subsets of $\Omega$, and
$G$ is a transitive permutation group on $\Omega$ such that $|Ag\cap B|\le1$
for all $g\in G$, then $|A|\cdot|B|\le|\Omega|$. We say that $G$ is
\emph{non-separating} if there exist sets $A$ and $B$, with $|Ag\cap B|=1$
for all $g\in G$, and $|A|\cdot|B|=|\Omega|$; we say that the pair $(A,B)$  witnesses non-separation. We say that $G$ is \emph{separating} otherwise. 

There is an analogue of the second part of the above result. The 
\emph{independence number} of a graph is the size of the largest induced
null graph (the maximum number of pairwise non-adjacent vertices).

\begin{theorem}\label{thrm:nonsep}
The transitive permutation group $G$ on $\Omega$ is non-separating if and only
if there is a graph $\Gamma$ on the vertex set $\Omega$, not complete or null,
such that the product of its clique number and independence number is equal
to $|\Omega|$, and $G\le\Aut(\Gamma)$.
\end{theorem}

If $G$ is transitive and $(A,P)$ witnesses non-synchronization, then $(A,B)$
witnesses non-separation for any part $B$ of $P$. (For, by the result cited
before Theorem~\ref{thrm:nonsep}, if $B$ is the largest part of $P$, then
\[|\Omega|\ge|A|\cdot|B|=|P|\cdot|B|\ge|\Omega|;\]
thus equality holds, which implies that all parts have the same size and
$|A|\cdot|B|=|\Omega|$ for any part.)
Thus, separation implies
synchronization. We are interested in the converse. Apart from four sporadic
examples (namely the symmetric and alternating groups of degree $10$ acting on
$4$-subsets, see~\cite[Section~5]{BC}, and $G_2(2)$ and its subgroup of index
$2$ with degree $63$), only one infinite family of primitive groups are known to
be synchronizing but not separating: these are the five-dimensional orthogonal
groups over finite fields of odd prime order, acting on the corresponding
quadrics; the proof of synchronization uses a result of Ball, Govaerts and
Storme~\cite{bgs} on ovoids on these quadrics. See~\cite[Section~6.2]{monster}.

It is easier to test for separation than for synchronization, since clique
number is easier to find in practice than chromatic number. Our main result
shows that this easier test suffices for synchronization except in the case
of almost simple groups.

For further information on these concepts we refer to the paper~\cite{monster}.

Primitive permutation groups are described by the O'Nan--Scott theorem, for
which we refer to Dixon and Mortimer~\cite{dm}. We need only a weak form of
the theorem:

\begin{theorem}
Let $G$ be a primitive permutation group on $\Omega$. Then one of the following
occurs:
\begin{itemize}
\item[(a)] $G$ is contained in a wreath product $H\wr K$ with product action,
and preserves a Cartesian decomposition of $\Omega$;
\item[(b)] $G$ is of affine, simple diagonal or almost simple type.
\end{itemize}
\end{theorem}

The Cartesian decompositions in Case~(a) are defined and studied in detail in
\cite[Chapter 8]{ps}. The simplest description for our purpose is that $G$ is
contained in the automorphism group of a \emph{Hamming graph} $H(n,q)$, whose
vertices are the $n$-tuples over an alphabet $A$ of size $q$, two vertices
adjacent if they agree in all but one coordinate. The clique number of the graph
is $q$: the set of $n$-tuples with fixed values in the first $n-1$ coordinates
is a clique. Moreover, if $A$ is an abelian group, then colouring an $n$-tuple
with the sum of its elements gives a proper $q$-colouring. So groups in
Case~(a) are not synchronizing.

In Case~(b), groups of affine type consist of mappings of the form
$x\mapsto xA+b$ on a vector space over a finite field of prime cardinality, 
where $A$ is an invertible linear map and $b$ a fixed vector; the socle of
such a group is the
\emph{translation subgroup} $\{x\mapsto x+b\}$. Groups of diagonal type are
described in more detail in the next section. Finally, $G$ is \emph{almost
simple} if $T\le G\le\Aut(T)$ for some non-abelian simple group $T$ (the
action of $T$ is not specified in this case).

\smallskip

Now we can state the main result.

\begin{theorem}
Let $G$ be a primitive permutation group which is not almost simple. Then
$G$ is synchronizing if and only if it is separating.
\end{theorem}

This follows immediately from the next two theorems.

\begin{theorem}
Let $G$ be a primitive permutation group of simple diagonal type, with more
than two factors in the socle. Then $G$ is non-synchronizing (and hence
non-separating).
\label{t:d3}
\end{theorem}

The proof of this theorem requires the Hall--Paige conjecture; the statement
of the conjecture, and the final case in its proof (for the sporadic simple
group $\jan{4})$, are given in Section~\ref{s:hp} below.

\begin{theorem}
Let $G$ be a primitive permutation group which is either of affine type,
or of simple diagonal type with two factors in the socle. Then $G$ is
synchronizing if and only if it is separating.
\label{t:ad2}
\end{theorem}

We remark that the affine case of this result is in~\cite{cz}: the proof
given below is a generalisation of the proof in~\cite{cz}.

\section{The Hall--Paige conjecture}
\label{s:hp}

\subsection{Preliminaries}

A \emph{complete mapping} on a group $G$ is a bijective function $\phi:G\to G$
such that the function $\psi:G\to G$ given by $\psi(g)=g\phi(g)$ is also a
bijection.

\begin{theorem}
A finite group $G$ has a complete mapping if and only if its Sylow $2$-subgroups
are not cyclic.
\label{t:hp}
\end{theorem}

This was conjectured by Hall and Paige \cite{hp}, who proved (among other
things) the  necessity of the condition, and showed its sufficiency for
alternating groups. Wilcox~\cite{wilcox} reduced the conjecture
to the case of simple groups, and proved it for groups of Lie type except for
the Tits group. Evans~\cite{evans} handled the Tits group and all the sporadic
groups except $\jan{4}$. The proof in the final case was announced by the first
author; we give details here.

We note in passing that the existence of a complete mapping for $G$ is
equivalent to the existence of an orthogonal mate of the Latin square which
is the Cayley table of $G$.

As well as completing the proof of the Hall--Paige conjecture, this section
is also an example of how it is possible to compute collapsed adjacency for a
permutation group of rather large degree (more than $10^9$).

Our main tool is the following \cite[Corollary~15]{wilcox}:

\begin{prop}
Let $G$ be a group having a subgroup $H$ which has a complete mapping. 
Let $\mathcal{D}$ be the set of double cosets $HgH$ of $H$ in $G$.
Suppose that there exist bijections $\phi,\psi:\mathcal{D}\to\mathcal{D}$
such that $|D|=|\phi(D)|=|\psi(D)|$ and $\psi(D)\subseteq D\phi(D)$ for all
$D\in\mathcal{D}$. Then $G$ has a complete mapping.
\label{p:w15}
\end{prop}

Note that the hypothesis about $\phi$ and $\psi$ is satisfied if it is the
case that every double coset except possibly $H$ has a representative of
order~$3$. For if $t$ is such a representative, then $t^{-1}=t^2\in D^2$,
and we can take $\phi(D)=D$ and $\psi(D)=D^{-1}$ for all $D\in\mathcal{D}$.

In more graph-theoretic terms, $G$ acts on the set of right cosets of $H$ by
right multiplication; each double coset $D$ corresponds to an orbital
graph $\Gamma$, where $D$ maps the point fixed by $H$ into its neighbourhood
in $\Gamma$; so if $t\in D$ has order $3$ and maps $x$ to $y$, then $t$ has
a $3$-cycle $(x,y,z)$, where the edges $(x,y)$, $(y,z)$, $(z,x)$ all belong to
$\Gamma$.

In fact, \cite[Corollary 16]{wilcox} gives a simpler sufficient condition,
namely $D\subseteq D^2$ for every double coset $D$ (taking $\phi$ and $\psi$
to be the identity maps); we will gain enough information to use this version
as an alternative.

The maximal subgroups of $\jan{4}$ are determined in~\cite{kw}.
From now on, we take $G=\jan{4}$, and $H$ the maximal subgroup
$\xsp{+}{12}.(3\udot \mat{22}\cn 2)$ (the third in the list in~\cite{kw}, and
the second
in the \wwwatlas \cite{www-atlas}, from which information about the group $G$
will be taken). Note that the existence of a complete mapping of $H$ follows
from the earlier results of \cite{evans,wilcox}.

Now $H$ is the full centraliser in $G\cong\jan{4}$ of a $2\AAA$-involution
($\jan{4}$-class), $x$ say, so that the actions of $G$ on the right cosets of
$H$ and on the conjugates of $x$ are isomorphic, with $Hg$ corresponding to
$x^g$. We shall consider our permutation action as a conjugation on
$2\AAA$-involutions from now on.

\subsection{Investigating the representation}

We first use character theory to obtain some basic information about our
permutation representation of degree $3980549947$: in particular, its rank
(the number of orbitals) and the number of self-paired orbitals. For this,
the tool is character theory. Fortunately \GAP~\cite{GAP} stores character tables for both
the groups $G\cong\jan{4}$ and $H\cong\xsp{+}{12}.(3\udot \mat{22}\cn 2)$.
Using \GAP, we find that there are 128 possible class fusions of $H$ into $G$, but that they
all give rise to the same permutation character, namely:
$$
\begin{array}{r@{\:+\:}l}
1 & 889111 + 1776888 + 4290927 + 35411145a^2b^2 + 95288172 \\
  & 230279749 + 259775040ab + 460559498 + 493456605 + 1016407168ab.
\end{array}
$$
In the above characters have been labelled with their degrees, with distinguishing letters if necessary,
and exponents denote multiplicity. All the above characters are integer valued except:
$$
\begin{array}{cc}
\mbox{Character} & \mbox{Irrationalities} \\ \hline
35411145a/b & \frac{1}{2}(1\pm \sqrt{33}) \\
259775040a/b & \pm 2\sqrt{3} \\
1016407168a/b & -1\pm2\sqrt{5}, \pm\sqrt{5} \\
\end{array}
$$
Thus the permutation character is a sum of 16 real characters, all with indicator $+$,
consisting of 12 characters occurring just once, and 2 characters that have multiplicity 2.
From general theory, the rank of
this permutation action is $20$ (the inner product of the permutation character with itself),
and there are $16$ self-paired orbitals (this is $\sum_i \ind(\chi_i)$, where the
permutation character is $\sum_i \chi_i$ and $\ind(\chi)$ denotes the Frobenius--Schur
indicator of $\chi$; here $\ind(\chi_i)$ is equal to $1$ for each of the 16 values of $i$).
Thus two pairs of non-self-paired orbitals which are not self-paired.

\subsubsection{Structure constant investigation}

The arguments in this (subsub-)section are not strictly necessary for the proof of the
Hall--Paige conjecture, but were used in the initial investigation of the problem.
If $x,y,x',y'$ are $2\AAA$-involutions and the pairs $(x,y)$ and $(x',y')$ are conjugate,
say $(x,y)^g=(x',y')$, then the elements $xy$ and $x'y'$ are conjugate (by $g$).
We now use (in \GAP) symmetrised structure constants to determine what classes are possible for $xy$.

For a group $G$, given classes $C_1=g_1^G,\, C_2=g_2^G, \ldots,\, C_n=g_n^G$ (where $g_1,\ldots,g_n$ are
arbitrary and repetitions are allowed), we define
$$
\hat{\xi}_G(C_1,C_2,\ldots,C_n)=
\frac{|G|^{n-1}}{|\cen_G(g_1)||\cen_G(g_2)|\cdots|\cen_G(g_n)|}
\sum_{\chi\in\mathrm{Irr}(G)} \frac{\chi(g_1)\chi(g_2)\cdots\chi(g_n)}{\chi(1)^{n-2}},
$$
which is the number of $n$-tuples $(x_1,\ldots,x_n)\in C_1\times \cdots \times C_n$ such that
$x_1\cdots x_n=1$. In practice, we prefer to count conjugacy classes of such tuples, and we have:
\begin{align*}
\xi_G(C_1,C_2,C_3)&=
\frac{|G|}{|\cen_G(g_1)|{\;\!\! . \;\!\!}|\cen_G(g_2)|{\;\!\! . \;\!\!}|\cen_G(g_3)|}
\sum_{\chi\in\mathrm{Irr}(G)} \frac{\chi(g_1)\chi(g_2)\chi(g_3)}{\chi(1)}\\
&=
\sum \frac{1}{|\cen_G(y_1,y_2,y_3)|},
\end{align*}
where the latter sum is taken over conjugacy class representatives of triples
$(y_1,y_2,y_3)\in C_1\times C_2 \times C_3$ such that $y_1y_2y_3=1$, and
$C_G(y_1,y_2,y_3)$ is the set of elements centralising each of $y_1,y_2,y_3$.
The structure constant calculations yield the information in Table~\ref{StrCon}.

\begin{table}[ht]
\begin{center}
\caption{The $(2\AAA,2\AAA,C)$ structure constants in $\jan{4}$}\label{StrCon}\vspace{5pt}
\begin{tabular}{c|c|c}
$C$ & $|H|\xi_{\jan{4}}(2\AAA,2\AAA,C)$ & $\xi_{\jan{4}}(2\AAA,2\AAA,C)$ \\ \hline
1A    &          1 &    $\frac{1}{21799895040}$ \\[3pt]
2A    &     112266 & $\frac{27}{5242880}$ \\[3pt]
2B    &      81840 & $\frac{31}{8257536}$ \\[3pt]
3A    &    8110080 & $\frac{1}{2688}$ \\[3pt]
4A    &     887040 & $\frac{1}{24576}$ \\[3pt]
4B    &   70963200 &   $\frac{5}{1536}$ \\[3pt]
4C    &   14192640 &   $\frac{1}{1536}$ \\[3pt]
5A    &  113541120 &   $\frac{1}{192}$ \\[3pt]
6B    &  340623360 &   $\frac{1}{64}$ \\[3pt]
6C    &   56770560 &   $\frac{1}{384}$ \\[3pt]
8C    &  340623360 &   $\frac{1}{64}$ \\[3pt]
10A   &  681246720 &   $\frac{1}{32}$ \\[3pt]
11B   &  990904320 &   $\frac{1}{22}$ \\[3pt]
12B   & 1362493440 &   $\frac{1}{16}$ \\[3pt]
%1A    &          1 &    1/21799895040 \\
%2A    &     112266 & 27/5242880 \\
%2B    &      81840 & 31/8257536 \\
%3A    &    8110080 & 1/2688 \\
%4A    &     887040 & 1/24576 \\
%4B    &   70963200 &   5/1536 \\
%4C    &   14192640 &   1/1536 \\
%5A    &  113541120 &   1/192 \\
%6B    &  340623360 &   1/64 \\
%6C    &   56770560 &   1/384 \\
%8C    &  340623360 &   1/64 \\
%10A   &  681246720 &   1/32 \\
%11B   &  990904320 &   1/22 \\
%12B   & 1362493440 &   1/16 \\
other &          0 &     0 \\
%% 1    &   1A    &   3980549947 &    86775571046077562880  &  1   &    1/21799895040 \\
%% 2    &   2A    &   194107 & 21799895040   &  112266  & 27/5242880 \\
%% 3    &   2B    &   69883 &  1816657920    &  81840   & 31/8257536 \\
%% 4    &   3A    &   1156  &  2661120 & 8110080 & 1/2688 \\
%% 5    &   4A    &   3323  &  5406720 & 887040  & 1/24576 \\
%% 6    &   4B    &   587   &  98304  & 70963200     &   5/1536 \\
%% 7    &   4C    &   259   &  43008  & 14192640     &   1/1536 \\
%% 8    &   5A    &   7     &  6720   & 113541120    &   1/192 \\
%% 10   &   6B    &   52    &  2304   & 340623360    &   1/64 \\
%% 11   &   6C    &   52    &  2304   & 56770560     &   1/384 \\
%% 16   &   8C    &   23    &  512    & 340623360    &   1/64 \\
%% 17   &   10A   &   7     &  960    & 681246720    &   1/32 \\
%% 20   &   11B   &   0     &  242    & 990904320    &   1/22 \\
%% 22   &   12B   &   8     &  192    & 1362493440   &   1/16 \\
\end{tabular}
\end{center}
\end{table}

Some of these rows correspond to more than one orbital, since there are only
fourteen non-zero rows. Paired orbitals are represented by the same row, and
Lemma 1.1.3 of \cite{kw} gives a splitting of the rows corresponding to
$2\AAA$ and $2\BB$ involutions. Further investigations in the group allowed
a complete splitting of the rows into orbitals. In particular, the splittings
of the $2\AAA$ and $2\BB$ rows are
$112266 = 1386+110880$ and $81840 = 18480 + 63360$. This gives us the smallest
orbitals, which are useful for further computation.

%{\scriptsize
%\begin{verbatim}
%j4:=CharacterTable("J4");ij4:=Irr(j4);;
%j42:=CharacterTable("J4m2");ij42:=Irr(j42);;
%ij42[1]; ##Check it's the trivial character.
%po:=PossibleClassFusions(j42,j4);;ll:=Length(po);
%pe:=List([1..ll],i->Induced(j42,j4,[Irr(j42)[1]],po[i])[1]);;pe:=Set(pe);;Length(pe);
%
%MatScalarProducts(pe,pe);
%MatScalarProducts(ij4,pe);
%PermCharInfo(j4,pe).ATLAS;
%## [ "1a+889111a+1776888a+4290927a+35411145aabb+
%95288172a+230279749a+259775040ab+460559498a+493456605a+1016407168ab" ]
%
%oo:=OrdersClassRepresentatives(j4);
%cj4:=SizesCentralizers(j4);
%for i in [1..Length(Irr(j4))] do aa:=ClassStructureCharTable(j4,[2,2,i])/Size(j4);
%if aa > 0 then Print(i,"\t",oo[i],"\t",pe[1][i],"\t",cj4[i],"\t",aa,"\n");fi;od;
%for i in [1..Length(Irr(j4))] do aa:=ClassStructureCharTable(j4,[2,2,i])/Size(j4);
%if aa > 0 then Print(i,"\t",oo[i],"\t",pe[1][i],"\t",cj4[i],"\t",aa*cj4[2],"\n");fi;od;
%for i in [1..Length(Irr(j4))] do aa:=ClassStructureCharTable(j4,[2,2,i])/Size(j4);
%if aa > 0 then Print(i,"\t",oo[i],"\t",pe[1][i],"\t",cj4[i],"\t",aa*cj4[2],"\t",aa,"\n");fi;od;
%\end{verbatim}
%}%

\subsection{Working in $\jan{4}$}

We have to choose a representation in which to do the computations, which must
be not too large and must allow  us to distinguish the orbitals with ease.

By far the most convenient representation for computational purposes turns out to be the (irreducible) 112-dimensional
representation of $\jan{4}$ over $\BBF_2$, which happens to be the smallest
representation in any characteristic. (The smallest faithful representation of $\jan{4}$ in odd
characteristic is $1333$, the same as in characteristic $0$, and the next smallest irreducible representation(s)
in characteristic 2 probably have degree $1220$. There are also no non-split modules with composition factors of dimensions $1$ and $112$, or $112$ and $1$.)
Such a representation of $\jan{4}$ is available from the \wwwatlas~\cite{www-atlas}.

Given a pair $(x,y)$ of involutions, we define the subspaces $V_i$ of $V:=\BBF_2^{112}$ as follows:
$V_0=V$, and for $i>0$ we have $V_{i+1}:=\langle V_i(1-x),V_i(1-y)\rangle = V_i(1-x)+V_i(1-y)$.
One easily proves (using induction) that if $V_i'$ is similarly defined starting from the pair
$(x',y')=(x,y)^g=(x^g,y^g)$ then $V_i'=V_i^g(=V_i.g)$ for all $i\in \mathbb{N}$.
Thus $d_i:=\dim V_i$ is an invariant of the conjugacy class of pairs $(x,y)$ of $2\AAA$-involutions.
Similarly, the dimensions $d'_1:=\dim(V(1-x)+V(1-yxy))$ and $d'_2:=\dim(V(1-y)+V(1-xyx))$ are
also invariants of conjugacy classes of pairs $(x,y)$ of $2\AAA$-involutions.
Only the invariants $d'_1$ and $d'_2$ are capable of distinguishing an orbital from its pair.
It turns out that the invariants $(d_1,d_2,d'_1,d'_2)$ suffice to distinguish all
the orbitals.

In order for our results to be reproducible, it is necessary to represent group elements
in terms of a tuple of `standard' generators. This consists of a generating tuple together
with some conditions that specify the tuple up to automorphism; in the case of $\jan{4}$
this means up to conjugacy, since all automorphisms of $\jan{4}$ are inner.
We use Rob Wilson's [Type I] standard generators of $\jan{4}$ given in the \wwwatlasn,
since we consider these as easy to find as any others. These are defined to be $a$ and $b$,
where $a$ is in class $2\AAA$, $b$ is in class $4\AAA$, $ab$ has order $37$ and $abab^2$ has order $10$.
The `black box' algorithm in the \wwwatlas suggests how finding the standard generators can be achieved.

The \wwwatlas supplies matrices for the $112$-dimensional $\BBF_2$-representation
of $\jan{4}$ on standard generators $a$ and $b$. We define further elements as follows:
$$
t:=(ab^2)^4,\quad c:=ab\quad\mbox{and}\quad d:=ba.
$$
We note that $t$ has order $3$, and that $c$ and $d$ are elements of a conveniently large order,
in this case $37$.

We now searched for representatives of all the orbitals, using the fingerprints given above,
some of which took quite some finding.
The information is summarised in Table~\ref{MainInfo},
which gives information on representatives of orbitals of $2\AAA$ involutions.

For a representative $(x,y)$ of each orbital the following information is displayed.
The numbers $d_1$, $d_2$, $d'_1$ and $d'_2$ are the above dimensions;
$s_1:=|y^{\cen_G(x)}|$, $s_2:=|\cen_G(\langle x,y\rangle)|$, so that
$s_1s_2=|\cen_G(x)|=21799895040$; `class' is the conjugacy class of $xy$ (in $\jan{4}$);
$t_i$ is an element such that $(a,a^{t_i})$ is a representative of orbital $i$;
and `pair' gives the number of the paired orbital of the current orbital (when different).
Note that at this stage we do not need the values of $s_1$ and $s_2$ given in the table. We will see later how these numbers can be computed.

%
%% T:=[a^2] cat [th^u:u in [ab^3*ba^3*ab^21*ba^12,ab^12*ba^8, ab^2*ba^9*ab^10*ba^5,ab^2*ba^11*ab^8*ba^3, ab^0,
%% ab^3*ba^10*ab^34, ab^12*ba^31,ab^6*ba^27, ab^2*ba^27,ab^4*ba^2, ab^5*ba^29,ab^2*ba^35,
%% ab^7,ab^3,ab^8,ab^2*ba^6,ab^5,ab,ab^2]];
%

We observe from the table that all double coset representatives apart from the
identity are conjugates of $t$, and so have order $3$; thus the conditions
of Proposition~\ref{p:w15} (\cite[Corollary 15]{wilcox}) with $\phi(D)=D$ and
$\psi(D)=D^{-1}$ are satisfied.

\begin{table}[ht]
\begin{center}
\caption{Information on representatives of orbitals of $2\AAA$ involutions}\label{MainInfo}\vspace{5pt}
\begin{tabular}{cc|c|ccccc|rl}
Nr & pair & $t_i$ & class & $d_1$ & $d_2$ & $d'_1$ & $d'_2$ & $s_1$ & $s_2$ \\ \hline
1 & self & identity & 1A & 50 & 0 & 50 & 50 & 1 & 21799895040 \\
2 & self & $t^{c^3d^3c^{21}d^{12}}$ & 2A & 72 & 16 & 50 & 50 & 1386 & 15728640 \\
3 & self & $t^{c^{12}d^{8}}$ & 2A & 75 & 20 & 50 & 50 & 110880 & 196608 \\
4 & self & $t^{c^2d^9c^{10}d^5}$ & 2B & 76 & 20 & 50 & 50 & 18480 & 1179648 \\
5 & self & $t^{c^2d^{11}c^8d^3}$ & 2B & 78 & 22 & 50 & 50 & 63360 & 344064 \\
6 & self & $t$ & 3A & 86 & 72 & 86 & 86 & 8110080 & 2688 \\
7 & self & $t^{c^3d^{10}c^{34}}$ & 4A & 88 & 56 & 72 & 72 & 887040 & 24576 \\
8 & 9 & $t^{c^{12}d^{31}}$ & 4B & 89 & 58 & 72 & 75 & 3548160 & 6144 \\
9 & 8 & $t^{c^{6}d^{27}}$ & 4B & 89 & 58 & 75 & 72 & 3548160 & 6144 \\
10 & self & $t^{c^2d^{27}}$ & 4B & 90 & 59 & 75 & 75 & 21288960 & 1024 \\
11 & self & $t^{c^4d^2}$ & 4B & 90 & 60 & 75 & 75 & 42577920 & 512 \\
12 & 13 & $t^{c^5d^{29}}$ & 4C & 91 & 63 & 76 & 78 & 7096320 & 3072 \\
13 & 12 & $t^{c^2d^{35}}$ & 4C & 91 & 63 & 78 & 76 & 7096320 & 3072 \\
14 & self & $t^{c^7}$ & 5A & 94 & 88 & 94 & 94 & 113541120 & 192 \\
15 & self & $t^{c^3}$ & 6B & 95 & 76 & 86 & 86 & 340623360 & 64 \\
16 & self & $t^{c^8}$ & 6C & 96 & 76 & 86 & 86 & 56770560 & 384 \\
17 & self & $t^{c^2d^6}$ & 8C & 98 & 82 & 90 & 90 & 340623360 & 64 \\
18 & self & $t^{c^5}$ & 10A & 99 & 88 & 94 & 94 & 681246720 & 32 \\
19 & self & $t^{c}$ & 11B & 100 & 100 & 100 & 100 & 990904320 & 22 \\
20 & self & $t^{c^2}$ & 12B & 100 & 90 & 95 & 95 & 1362493440 & 16 \\
\end{tabular}
\end{center}
\end{table}

\subsection{Collapsed adjacency matrices for this action}

Now that we can identify the orbital that contains any pair $(g,h)$ of $2\AAA$
involutions of $G\cong \jan{4}$, we are in a position to calculate the
collapsed adjacency matrices associated with this action
for various orbitals. The notation $G,a,b,c,d,t,t_i$ is as in previous sections.

First of all, we need to obtain $\cen_G(a)$, for which we use standard methods
\cite{jnb}. We get
$$
H=\cen_G(a) = \langle a, [a,b]^5, (ab^2)^6, babab[a,babab]^5, bab^2ab[a,bab^2ab]^5 \rangle,
$$
or, if we insist on just two generators, we can take
$$
H=\cen_G(a) = \langle [a,b]^5(ab^2)^6,  bab^2ab[a,bab^2ab]^5 ababab[a,ababab]^5 \rangle.
%%CA2:=sub<G|(a,b)^5*(a*b^2)^6,b*a*b^2*a*b*(a,b*a*b^2*a*b)^5*ab^3*(a,ab^3)^5>; // C_G(a).
$$
We show that the above groups are subgroups of $\cen_G(a)$, simply by showing 
that generators of the subgroups centralise $a$.
We used {\sc Magma}~\cite{Magma} to verify that the second group above is
indeed the whole of $\cen_G(a)$, by computing its order.

The neighbourhood of $a$ in the $i$-th orbital graph is the orbit of $a^{t_i}$
under $\cen_G(a)$, which is found by closing $\{a^{t_i}$ by repeatedly
conjugating by the generators $h_1=[a,b]^5(ab^2)^6$ and
$h_2=bab^2ab[a,bab^2ab]^5 ababab[a,ababab]^5$ of $\cen_G(a)$.
Call this orbit $O_i$.

We then obtain the $i$-th collapsed adjacency matrix $A_i$ as follows.
For each value of $j$ and each element $y$ of $O_i^{t_j}$ we determine which orbital $(a,y)$ belongs to,
using the fingerprints given above. The $(j,k)$ entry of $A_i$ is then the number of $y\in O_i^{t_j}$
for which this orbital is the $k$-th. This is the number of paths
$(a,y,a^{t_j})$ of type $(O_k,O_{i^*})$ on a fixed base $(a,a^{t_j})$ of type
$O_j$, where $O_{i^*}$ is paired with $O_i$.

In fact, there are memory issues using this method to calculate all the $A_i$.
It turns out to be enough to calculate $A_2$ and $A_4$, which correspond to the two smallest
non-trivial orbitals. This computation is quite fast.
Recently we have also calculated $A_5$ by this method to provide a check
on our work, but this was not done originally.

The reason that computation of $A_2$ and $A_4$ suffice is that the
\textit{intersection algebra} generated by $A_2$ and $A_4$ has dimension $20$
and contains all
the collapsed adjacency matrices $A_i$, which occur as scalar multiples of
the natural basis elements of this algebra. In each case, the first row
and first column of the basis elements have weight $1$, and we scale them
so that the non-zero entry in the first column is $1$ (they are given so that
the non-zero entry in the first row is $1$).

The collapsed adjacency matrix corresponding to the second (and also second smallest) suborbit, of
size $1386$, is given below.
{\scriptsize
\def\arraycolsep{5pt}
$$
\begin{array}{rrrrrrrrrrrrrrrrrrrr}
0 & 1386 & 0 & 0 & 0 & 0 & 0 & 0 & 0 & 0 & 0 & 0 & 0 & 0 & 0 & 0 & 0 & 0 & 0 & 0 \\
1 & 65 & 240 & 120 & 320 & 0 & 640 & 0 & 0 & 0 & 0 & 0 & 0 & 0 & 0 & 0 & 0 & 0 & 0 & 0 \\
0 & 3 & 75 & 12 & 16 & 0 & 96 & 0 & 224 & 192 & 384 & 384 & 0 & 0 & 0 & 0 & 0 & 0 & 0 & 0 \\
0 & 9 & 72 & 57 & 0 & 0 & 288 & 576 & 0 & 0 & 0 & 384 & 0 & 0 & 0 & 0 & 0 & 0 & 0 & 0 \\
0 & 7 & 28 & 0 & 35 & 0 & 196 & 0 & 112 & 672 & 0 & 0 & 336 & 0 & 0 & 0 & 0 & 0 & 0 & 0 \\
0 & 0 & 0 & 0 & 0 & 21 & 7 & 14 & 0 & 84 & 0 & 0 & 28 & 112 & 84 & 28 & 0 & 336 & 0 & 672 \\
0 & 1 & 12 & 6 & 14 & 64 & 57 & 80 & 48 & 192 & 96 & 96 & 144 & 0 & 384 & 192 & 0 & 0 & 0 & 0 \\
0 & 0 & 0 & 3 & 0 & 32 & 20 & 83 & 0 & 0 & 0 & 48 & 48 & 192 & 576 & 0 & 0 & 384 & 0 & 0 \\
0 & 0 & 7 & 0 & 2 & 0 & 12 & 0 & 45 & 72 & 240 & 32 & 48 & 0 & 192 & 64 & 288 & 0 & 0 & 384 \\
0 & 0 & 1 & 0 & 2 & 32 & 8 & 0 & 12 & 75 & 56 & 16 & 32 & 0 & 256 & 64 & 0 & 192 & 0 & 640 \\
0 & 0 & 1 & 0 & 0 & 0 & 2 & 0 & 20 & 28 & 71 & 16 & 0 & 32 & 128 & 32 & 192 & 288 & 256 & 320 \\
0 & 0 & 6 & 1 & 0 & 0 & 12 & 24 & 16 & 48 & 96 & 63 & 0 & 48 & 192 & 208 & 0 & 288 & 0 & 384 \\
0 & 0 & 0 & 0 & 3 & 32 & 18 & 24 & 24 & 96 & 0 & 0 & 69 & 128 & 384 & 32 & 192 & 0 & 0 & 384 \\
0 & 0 & 0 & 0 & 0 & 8 & 0 & 6 & 0 & 0 & 12 & 3 & 8 & 33 & 120 & 8 & 96 & 228 & 384 & 480 \\
0 & 0 & 0 & 0 & 0 & 2 & 1 & 6 & 2 & 16 & 16 & 4 & 8 & 40 & 155 & 24 & 144 & 280 & 288 & 400 \\
0 & 0 & 0 & 0 & 0 & 4 & 3 & 0 & 4 & 24 & 24 & 26 & 4 & 16 & 144 & 81 & 48 & 288 & 192 & 528 \\
0 & 0 & 0 & 0 & 0 & 0 & 0 & 0 & 3 & 0 & 24 & 0 & 4 & 32 & 144 & 8 & 163 & 192 & 384 & 432 \\
0 & 0 & 0 & 0 & 0 & 4 & 0 & 2 & 0 & 6 & 18 & 3 & 0 & 38 & 140 & 24 & 96 & 255 & 352 & 448 \\
0 & 0 & 0 & 0 & 0 & 0 & 0 & 0 & 0 & 0 & 11 & 0 & 0 & 44 & 99 & 11 & 132 & 242 & 363 & 484 \\
0 & 0 & 0 & 0 & 0 & 4 & 0 & 0 & 1 & 10 & 10 & 2 & 2 & 40 & 100 & 22 & 108 & 224 & 352 & 511 \\
\end{array}
$$
}%
The collapsed adjacency matrix corresponding to the fourth (but third smallest) suborbit, of
size $18480$, is given below.
{\scriptsize
\def\arraycolsep{3pt}
$$
\begin{array}{rrrrrrrrrrrrrrrrrrrr}
0 & 0 & 0 & 18480 & 0 & 0 & 0 & 0 & 0 & 0 & 0 & 0 & 0 & 0 & 0 & 0 & 0 & 0 & 0 & 0 \\
0 & 120 & 960 & 760 & 0 & 0 & 3840 & 0 & 7680 & 0 & 0 & 0 & 5120 & 0 & 0 & 0 & 0 & 0 & 0 & 0 \\
0 & 12 & 196 & 32 & 96 & 0 & 576 & 0 & 672 & 768 & 1536 & 1536 & 768 & 0 & 3072 & 3072 & 0 & 6144 & 0 & 0 \\
1 & 57 & 192 & 182 & 96 & 3072 & 1056 & 768 & 768 & 0 & 0 & 1536 & 1536 & 6144 & 0 & 3072 & 0 & 0 & 0 & 0 \\
0 & 0 & 168 & 28 & 168 & 0 & 420 & 448 & 1680 & 1344 & 1344 & 0 & 2128 & 0 & 10752 & 0 & 0 & 0 & 0 & 0 \\
0 & 0 & 0 & 7 & 0 & 98 & 21 & 126 & 0 & 84 & 168 & 140 & 84 & 448 & 2436 & 756 & 672 & 3360 & 4032 & 6048 \\
0 & 6 & 72 & 22 & 30 & 192 & 270 & 384 & 336 & 864 & 672 & 576 & 528 & 1536 & 4224 & 1088 & 0 & 4608 & 0 & 3072 \\
0 & 0 & 0 & 4 & 8 & 288 & 96 & 212 & 0 & 576 & 384 & 96 & 304 & 960 & 3264 & 384 & 3072 & 2688 & 3072 & 3072 \\
0 & 3 & 21 & 4 & 30 & 0 & 84 & 0 & 210 & 672 & 480 & 192 & 144 & 128 & 1728 & 192 & 1920 & 1920 & 3072 & 7680 \\
0 & 0 & 4 & 0 & 4 & 32 & 36 & 96 & 112 & 404 & 352 & 96 & 192 & 576 & 2048 & 576 & 1408 & 3200 & 3072 & 6272 \\
0 & 0 & 4 & 0 & 2 & 32 & 14 & 32 & 40 & 176 & 420 & 128 & 32 & 320 & 1824 & 352 & 1728 & 2816 & 4096 & 6464 \\
0 & 0 & 24 & 4 & 0 & 160 & 72 & 48 & 96 & 288 & 768 & 348 & 96 & 304 & 2304 & 624 & 1920 & 4128 & 3072 & 4224 \\
0 & 1 & 12 & 4 & 19 & 96 & 66 & 152 & 72 & 576 & 192 & 96 & 202 & 384 & 2304 & 480 & 384 & 3072 & 3072 & 7296 \\
0 & 0 & 0 & 1 & 0 & 32 & 12 & 30 & 4 & 108 & 120 & 19 & 24 & 574 & 1632 & 264 & 1632 & 3276 & 4608 & 6144 \\
0 & 0 & 1 & 0 & 2 & 58 & 11 & 34 & 18 & 128 & 228 & 48 & 48 & 544 & 1744 & 256 & 1584 & 2976 & 4544 & 6256 \\
0 & 0 & 6 & 1 & 0 & 108 & 17 & 24 & 12 & 216 & 264 & 78 & 60 & 528 & 1536 & 462 & 1200 & 3456 & 4032 & 6480 \\
0 & 0 & 0 & 0 & 0 & 16 & 0 & 32 & 20 & 88 & 216 & 40 & 8 & 544 & 1584 & 200 & 1716 & 3200 & 4608 & 6208 \\
0 & 0 & 1 & 0 & 0 & 40 & 6 & 14 & 10 & 100 & 176 & 43 & 32 & 546 & 1488 & 288 & 1600 & 3208 & 4576 & 6352 \\
0 & 0 & 0 & 0 & 0 & 33 & 0 & 11 & 11 & 66 & 176 & 22 & 22 & 528 & 1562 & 231 & 1584 & 3146 & 4796 & 6292 \\
0 & 0 & 0 & 0 & 0 & 36 & 2 & 8 & 20 & 98 & 202 & 22 & 38 & 512 & 1564 & 270 & 1552 & 3176 & 4576 & 6404 \\
\end{array}
$$
}%

We have determined the collapsed adjacency matrices of all the orbitals in
this way, and checked directly that the condition on double cosets
(namely $D^{-1}\subseteq D^2$) is satisfied. This involves checking that
$(A_i)_{ii^*}\ne0$ for each $i$, where $i^*$ is the number of the orbital
paired with $O_i$.The relevant entries are
\[\begin{array}{l}
1, 65, 1456, 182, 280, 32560, 3360, 5888, 5888,\\
126352, 464672, 18816, 18816, 3246240, 29201232,\\
780816, 29096448, 116607440, 246648576, 466371136
\end{array}\]
As noted earlier, once we have all the collapsed adjacency matrices, we can
verify the simpler condition $D\subseteq D^2$ of \cite[Corollary~16]{wilcox}
(Proposition~\ref{p:w15} with $\phi$ and $\psi$ both the identity map)
by checking that $(A_i)_{ii}\ne0$ for all $i$. The relevant numbers are those
of the above list with $5888$ and $18816$ replaced by $3648$ and $14592$
respectively.

Note also that the values of the entries $s_1$ and $s_2$ in Table~\ref{MainInfo}
can also be read from the collapsed adjacency matrices.

\section{Proofs of Theorems~\ref{t:d3} and \ref{t:ad2}}

\subsection{Diagonal groups with more than two socle factors}

In this section, we recall the definition of diagonal groups, and prove
Theorem~\ref{t:d3}.

First recall the diagonal group $D(T,n)$, where $T$ is a non-abelian simple
group and $n$ an integer greater than $1$. This is a permutation group on
the set
\[\Omega=\{(t_2,\ldots,t_n):t_2,\ldots,t_n\in T\}=T^{n-1},\]
and is generated by the following permutations of $\Omega$:
\begin{itemize}\itemsep0pt
\item[(G1)] $(s_1,\ldots,s_n):(t_2,\ldots,t_n)\mapsto(s_1^{-1}t_2s_2,\ldots,
s_1^{-1}t_ns_n)$ for $(s_1,\ldots,s_n)\in T^n$ (these form a group isomorphic to $T^n$, which is the
socle of $D(T,n)$);
\item[(G2)] $\alpha:(t_2,\ldots,t_n)\mapsto(t_2^\alpha,\ldots,t_n^\alpha)$ for
$\alpha\in\Aut(T)$ (the inner automorphisms of $T$ coincide with the
permutations $(s,\ldots,s)$ of the preceding type);
\item[(G3)] $\pi\in\Sym(\{2,\ldots,n\})$ acting on the coordinates of points in
$\Omega$;
\item[(G4)]
$\tau:(t_2,\ldots,t_n)\mapsto(t_2^{-1},t_2^{-1}t_3,\ldots,t_2^{-1}t_n)$
(this corresponds to the transposition $(1,2)$ in $S_n$; together with the
preceding type it generates a group isomorphic to $S_n$).
\end{itemize}
More details, and a characterisation, for diagonal groups will be given
in~\cite{bcps}.

We define a graph $\Gamma$ on the vertex set $\Omega$ by the rule that
$(t_2,\ldots,t_n)$ is joined to $(u_2,\ldots,u_n)$  if and only if one of the
following holds:
\begin{itemize}\itemsep0pt
\item[(A1)] there exists $i\in\{2,\ldots,n\}$ such that $u_i\ne t_i$ but
$u_j=t_j$ for $j\ne i$;
\item[(A2)] there exists $x\in T$ with $x\ne1$ such that $u_i=xt_i$ for
$i=2,\ldots,n$.
\end{itemize}

Showing that $D(T,n)\le\Aut(\Gamma)$ is just a matter of checking:
\begin{itemize}
\item Consider a generator of type (G1). This obviously preserves adjacency of
type (A1). For (A2), suppose that $u_i=xt_i$ for all $i$. Applying a map of the
first kind with $s_1=1$ obviously preserves adjacency, so we can suppose that
$s_2=\cdots=s_n=1$. Then
\[s_1^{-1}u_i=(s_1^{-1}xs_1)s_1^{-1}t_i=ys_1^{-1}t_i\]
with $y=s_1^{-1}xs_1$, so the
vertices are adjacent by the (A2) rule (using $y$ in place of $x$).
\item A generator of type (G2) clearly preserves both types of adjacency rule
(with $x^\alpha$ replacing $x$ in (A2)).
\item A generator of type (G3) also preserves both
adjacency rules.
\item It remains to check $\tau$. Suppose that $(t_2,\ldots,t_n)$ is
adjacent to $(u_2,\ldots,u_n)$. Suppose that the adjacency uses rule (A1)
with $i\ne2$. Then the two vertices are mapped to 
$(t_2^{-1},t_2^{-1}t_3,\ldots,t_2^{-1}t_n)$ and
$(u_2^{-1},u_2^{-1}u_3,\ldots,u_2^{-1}u_n)$; these agree in all coordinates
except the $i$th, and so are adjacent by the rule (A1). Suppose that the adjacency uses (A1)
with $i=2$. Then $t_3=u_3,\ldots,t_n=u_n$, but $t_2\ne u_2$. If
$u_2=t_2x$ with $x\ne1$, then the images are adjacent by rule (A2)
with $x^{-1}$ replacing $x$. Finally, suppose that the adjacency uses rule
(A2), so that $u_i=xt_i$ for all $i$. Then $u_2^{-1}=t_2^{-1}x^{-1}$ but
$u_2^{-1}u_i=t_2^{-1}t_i$ for $i>2$, so the vertices are adjacent by (A1), 
with $i=2$.
\end{itemize}

The neighbourhood of a vertex in $\Gamma$ is the disjoint union of $n$
cliques; $n-1$ of these are given by adjacencies of the first type with a
fixed value of $i$, and the last one by adjacencies of the second type.
If $n>3$, there are no edges between vertices of different cliques, so
$\Gamma$ has clique number $|T|$. This is also true when $n=3$, in which case
the graph is the \emph{Latin square graph} associated with the Cayley table
of $T$.

Note in passing that if we delete rule (A2), or else delete rule (A1) for a
fixed value of $i$, we obtain a graph isomorphic to the Hamming graph
$H(n-1,|T|)$.

Note also that, for $n>2$, the automorphism group of $\Gamma$ is actually
equal to $D(T,n)$. This fact is not required for our proof; it will be
proved in the forthcoming paper~\cite{bcps}.

To prove Theorem~\ref{t:d3}, we are going to show that, for $n>2$, there is a 
proper colouring of $\Gamma$ with $|T|$ colours. It will follow that $\Gamma$
has clique number equal to chromatic number, so that its automorphism
group (and in particular, the group $D(T,n)$ and any primitive subgroup of
it) is non-synchronizing.

We split the proof into two cases according as $n$ is even or odd.

\paragraph{Case $n$ even, $n>2$.} In this case, we define a colouring of the
vertex set of $D(T,n)$, with $T$ as the set of colours, as follows:
\begin{quote}
the colour of the vertex $(t_2,\ldots,t_n)$ is
$(t_2^{-1}t_3)(t_4^{-1}t_5)\cdots (t_{n-2}^{-1}t_{n-1})t_n^{-1}$.
\end{quote}
We must check that this is a proper colouring.
\begin{itemize}\itemsep0pt
\item For adjacencies of type (A1), adjacent vertices differ in just one
coordinate, and so clearly their colours differ.
\item Suppose that $(t_2,\ldots,t_n)$ is adjacent to $(u_2,\ldots,u_n)$ by
rule (A2), so $u_i=xt_i$ for all $i$, with $x\ne1$. Let $a$ be the colour of $(t_2,\ldots,t_n)$ and let $b$ be the colour of $(u_2,\ldots,u_n)$. Then
\[\begin{array}{rcl}
b&=&(u_2^{-1}u_3)(u_4^{-1}u_5)\cdots (u_{n-2}^{-1}u_{n-1})u_n^{-1}\\
&=&((t_2^{-1}x^{-1})(xt_3))
((t_4^{-1}x^{-1})(xt_5))\cdots
((t_{n-2}^{-1}x^{-1})(xt_{n-1}))(t_n^{-1}x^{-1})\\
&=&(t_2^{-1}t_3)(t_4^{-1}t_5)\cdots (t_{n-2}^{-1}t_{n-1})t_n^{-1}x^{-1}\\
&=& ax^{-1}\ne a,
\end{array}\]
so these vertices have different colours.
\end{itemize}

\paragraph{Case $n$ odd.} This case is more complicated, and requires the
truth of the Hall--Paige conjecture (Theorem~\ref{t:hp}). We note that, by
Burnside's transfer theorem, the Sylow $2$-subgroups of a non-abelian finite
simple group cannot be cyclic; so any such group has a complete mapping.

So let $\phi:T\to T$ be a complete mapping for $T$, and let $\psi:T\to T$ be the bijection defined by
$\psi(g)=g\phi(g)$. We define a colouring of the vertex set of $D(T,n)$ for
$n$ odd as follows:
\begin{quote}
the vertex $(t_2,\ldots,t_n)$ is given the colour
$(t_2^{-1}t_3)(t_4^{-1}t_5)\cdots (t_{n-3}^{-1}t_{n-2})(t_{n-1}^{-1}\psi(t_n))$.
\end{quote}	
We check that this is a proper colouring.
\begin{itemize}\itemsep0pt
\item Suppose two vertices are adjacent by rule (A1), with $i<n$. Then
they differ in the $i$th coordinate, and so their colours differ.
\item The same holds if $i=n$, since $\psi$ is a bijection.
\item Suppose that $(u_2,\ldots,u_n)$ is adjacent to $(t_2,\ldots,t_n)$ by
rule (A2): $u_i=xt_i$ for all $i$, with $x\ne1$. Let $a$ be the colour of $(t_2,\ldots,t_n)$ and let $b$ be the colour of $(u_2,\ldots,u_n)$. Then 
\[\begin{array}{rcl}
b&=&(u_2^{-1}u_3)(u_4^{-1}u_5)\cdots (u_{n-3}^{-1}u_{n-2})(u_{n-1}^{-1}\psi(u_n))\\
&=&((t_2^{-1}x^{-1})(xt_3))
((t_4^{-1}x^{-1})(xt_5))\cdots
((t_{n-3}^{-1}x^{-1})(xt_{n-2}))(t_{n-1}^{-1}x^{-1}\psi(xt_n))\\
&=&(t_2^{-1}t_3)(t_4^{-1}t_5)\cdots (t_{n-3}^{-1}t_{n-2})(t_{n-1}^{-1}\psi(t_n))\psi(t_n)^{-1}x^{-1}\psi(xt_n)\\
&=& a\psi(t_n)^{-1}x^{-1}\psi(xt_n).
\end{array}\]
So we need to show that 
$\psi(xt_n)\ne x\psi(t_n)$. Since $\psi(g)=g\phi(g)$, we have to show that
$xt_n\phi(xt_n)\ne xt_n\phi(t_n)$, which is true since $\phi$ is a bijection
and $x\ne1$.
\end{itemize}
The theorem is proved.

\subsection{Groups with regular subgroups}

In this section, we prove Theorem~\ref{t:ad2}. The simple argument is more
general; we consider synchronization and separation for permutation groups $G$
having a regular subgroup, and show that if $G$ contains both the left and
the right actions of this subgroup then the two concepts are equivalent.
We noted after Theorem~\ref{thrm:nonsep} that separation implies
synchronization; our business here is to show the converse, for affine groups
and for diagonal groups with two factors in the socle.

Let $G$ be a permutation group of degree $n$ with a regular subgroup $H$. 
Then $G$ can be represented as a permutation group on the set $H$: we choose
a point $\alpha\in\Omega$ to correspond to the identity, and identify $\beta$
with $h$ where $\alpha h=\beta$. Then $H$ acts on itself by right 
multiplication.

Recall that sets $A$ and $B$ witness non-separation if $|A|, |B|>1$,
$|A|\cdot|B|=n$, and $|Ag\cap B|=1$ for all $g\in G$; the set $A$ and
partition $P$ witness non-synchronization if $|A|>1$ and $Ag$ is a transversal
for $P$ for all $g\in G$.

\begin{prop}\label{pr1}
Suppose that $A$ and $B$ witness non-separation. Then $H$ has an exact
factorisation by $A^{-1}=\{a^{-1}:a\in A\}$ and $B$, that is, every element
of $H$ is uniquely expressible as $a^{-1}b$ for $a\in A$ and $b\in B$.
\end{prop}

\begin{proof} Since $|A^{-1}|\cdot|B|=|H|$, it is enough to show that
factorisation is unique. So suppose that $a_1^{-1}b_1=a_2^{-1}b_2$, where
$a_1,a_2\in A$ and $b_1,b_2\in B$. Then
\[\begin{array}{rcl}
b_1 &=& a_1a_2^{-1}b_2 \in A(a_2^{-1}b_2)\cap B,\\
b_2 &=& a_2a_1^{-1}b_1 \in A(a_1^{-1}b_1)\cap B,
\end{array}\]
so $b_1=b_2$ and $a_1=a_2$.
\end{proof}

\begin{prop}\label{pr2}
Suppose that $A$ and $B$ witness non-separation, and assume that $H$ has an
exact factorisation by $A$ and $B$. Then $G$ is non-synchronizing.
\end{prop}

\begin{proof} We claim that $P=\{Ab:b\in B\}$ is a partition of $H$. For,
if $x\in Ab_1\cap Ab_2$, then $x=a_1b_2=a_2b_2$ for some $a_1,a_2\in A$;
since $H$ has an exact factorization by $A$ and $B$, we get $a_1=a_2$ and $b_1=b_2$.

Now for any $g\in G$, $|Ab\cap Bg|=|A\cap Bgb^{-1}|=1$ because $A$ and $B$ witness non-separation, so $P$ and $B$ witness
non-synchronization.
\end{proof}

\begin{cor}
Let $G$ be a permutation group with a regular subgroup $H$. Suppose that $G$
contains both the right and the left action of $H$. Then $G$ is synchronising
if and only if it is separating. In particular, this is true if $H$ is abelian.
\label{cor:leftright}
\end{cor}

\begin{proof}
Suppose that $A$ and $B$ witness non-separation of $G$. By Theorem~\ref{thrm:nonsep}, there is
a graph $\Gamma$ with vertex set $H$ such that $G\le\Aut(\Gamma)$, and $A$ is a
clique and $B$ a coclique in $\Gamma$. Since $H$ acts regularly on the vertices of $\Gamma$, we deduce that $\Gamma$ is a Cayley graph over $H$, thus $\Gamma=\mathrm{Cay}(H,S)$ for some subset $S$ of $H$. Since $\Gamma$ admits the left and
right actions of $H$,  the connection set $S$ is closed under conjugation
in $H$. Now $A$ is a clique, so $a_1a_2^{-1}\in S$ for all $a_1,a_2\in A$; thus
also $a_2^{-1}a_1=a_2^{-1}(a_1a_2^{-1})a_2\in S$ for all $a_2,a_1\in A$, and $A^{-1}$ is also a clique.
The result now follows from Propositions~\ref{pr1} and~\ref{pr2}.
\end{proof}

Now we can deal with the remaining classes of primitive groups. Both are
immediate from Corollary~\ref{cor:leftright}. 

\paragraph{Affine groups} The socle of an affine group is the translation group
of the affine space, which is an abelian regular subgroup. The left and right
regular actions of an abelian group are the same.

\paragraph{Diagonal groups with two factors} The socle of such a group has
the form $T\times T$, where $T$ is a non-abelian simple group; it acts on $T$
by the rule $(g,h):x\mapsto g^{-1}xh$. So the first factor of $T\times T$
induces the left regular action of $T$, and the second factor the right
regular action.

\paragraph{Problem} Is it true that, for any group $G$ containing a regular
subgroup $H$, $G$ is synchronizing if and only if it is separating?

\paragraph{Problem} Is it true that every group of simple diagonal type with
two simple factors in its socle is non-synchronizing?

\paragraph{Acknowledgment} The research of Qi Cai and Hua Zhang was
supported by National Science Foundation of China grant 11561078. The
authors are grateful to the referee for a thorough and helpful report.


\begin{thebibliography}{99}


\bibitem{BC}Mohammed Aljohani, John Bamberg and Peter J.~Cameron,
Synchronization and separation in the Johnson schemes,
\textit{Port. Math. }\textbf{74} (2017), no. 3, 213--232. 

\bibitem{monster}Jo\~ao Ara\'ujo, Peter J. Cameron and Benjamin Steinberg,
Between primitive and 2-transitive: Synchronization and its friends,
\textit{Europ. Math. Soc. Surveys} \textbf{4} (2017), 101--184.


\bibitem{bcps}
Rosemary A. Bailey, Peter J. Cameron, Cheryl E. Praeger and Csaba Schneider,
Diagonal structures, in preparation.

\bibitem{bgs}
Simeon Ball, Patrick Govaerts and Leo Storme,
On ovoids of parabolic quadrics,
\textit{Designs, Codes, Cryptography} \textbf{38} (2006), 131--145.

\bibitem{Magma}
Wieb Bosma, John Cannon, and Catherine Playoust,
The Magma algebra system. I. The user language,
\textit{J. Symbolic Comput.} \textbf{24} (1997), 235--265. 

\bibitem{jnb}
J.N. Bray,
An improved method for generating the centralizer of an involution,
\textit{Arch. Math.} \textbf{74} (2000), 241--245.

\bibitem{cz}
Qi Cai and Hua Zhang,
Synchronizing groups of affine type are separating,
preprint.

\bibitem{dm}
John D. Dixon and Brian Mortimer,
\textit{Permutation Groups},
Springer-Verlag, new York, 1996.

\bibitem{evans}
Anthony B. Evans,
The admissibility of sporadic simple groups.
\textit{J. Algebra} \textbf{321} (2009), no. 1, 105--116.

\bibitem{GAP}
The GAP Group, GAP -- Groups, Algorithms, and Programming, Version 4.10.0; 2018.
(\url{https://www.gap-system.org}) 

\bibitem{hp}
Marshall Hall Jr. and L. J. Paige,
Complete mappings of finite groups,
\textit{Pacific J. Math.} \textbf{5} (1955), 541--549.

\bibitem{kw}
Peter B. Kleidman and Robert A. Wilson,
The maximal subgroups of $\jan{4}$,
\textit{Proc. London Math. Soc.} \textbf{56} (1988), 484--510.

\bibitem{ps}
Cheryl E. Praeger and Csaba Schneider,
\textit{Permutation Groups and Cartesian Decompositions},
London Math. Soc. Lecture Note Series \textbf{449},
Cambridge University Press, Cambridge, 2018.

\bibitem{wilcox}
Stewart Wilcox,
Reduction of the Hall--Paige conjecture to sporadic simple groups.
\textit{J. Algebra} \textbf{321} (2009), no. 5, 1407--1428.

\bibitem{www-atlas}
Robert Wilson, Peter Walsh, Jonathan Tripp, Ibrahim Suleiman, Richard Parker,
Simon Norton, Simon Nickerson, Steve Linton, John Bray, and Rachel Abbott,
\textsc{Atlas} of Finite Group Representations -- Version 3,
\url{http://brauer.maths.qmul.ac.uk/Atlas/v3/}

\end{thebibliography}
\end{document}